\documentclass{amsart}
\usepackage{mathrsfs}
\usepackage{amsfonts}
\usepackage{graphicx}
 \usepackage{epsfig}
 \usepackage[colorlinks=true]{hyperref}
\hypersetup{urlcolor=blue, citecolor=red}
\newtheorem{theorem}{Theorem}[section]
\newtheorem{lemma}[theorem]{Lemma}

\theoremstyle{definition}
\newtheorem{definition}[theorem]{Definition}
\newtheorem{proposition}[theorem]{Proposition}
\newtheorem{corollary}[theorem]{Corollary}
\newtheorem{example}[theorem]{Example}

\theoremstyle{remark}
\newtheorem{remark}[theorem]{Remark}
\def\imag{\operatorname{Im}}
\def\interior{\operatorname{Int}}
\numberwithin{equation}{section}

\begin{document}

\title{Lyapunov graphs of nonsingular Smale flows on $S^{1}\times S^{2}$}

\author{Bin Yu}
\address{Department of Mathematics, Tongji University, Shanghai, China 20092}

\thanks{The author was supported in part by the NSFC (grant no. 11001202).}


\date{?????  ?, 2010 and, in revised form, ????? ?, 201?.}


\begin{abstract}
 In this paper, following J. Franks' work on Lyapunov graphs
 of nonsingular Smale flows on $S^3$, we study Lyapunov graphs
 of nonsingular Smale flows on $S^1 \times S^2$. More
 precisely, we determine necessary and  sufficient conditions on an abstract Lyapunov graph
  to be  associated with a nonsingular Smale flow on $S^1 \times S^2$.
 We also study the singular type vertices in Lyapunov
  graphs of nonsingular Smale flows on 3-manifolds.
\end{abstract}
\maketitle

\section{Introduction} \label{section1}

Lyapunov graphs were first introduced in dynamics by J. Franks in
his paper \cite{F5}. For a smooth flow $\phi_{t}:M\rightarrow M$
with a Lyapunov function $f:M\rightarrow R$, a Lyapunov graph $L$ is
a rather natural object. The idea is to construct an oriented graph
by identifying to a point each component of $f^{-1}(c)$ for each
$c\in R$. A Lyapunov graph gives
     a global picture of how the basic sets of a smooth flow are
     situated on the underlying manifold. It is important to note
     that a Lyapunov graph  does not always determine a unique flow
     up to topological equivalence.

In \cite{F5}, J. Franks used Lyapunov graph to classify nonsingular
Smale flows (abbreviated as \textit{NS flows}) on $S^3$. More
precisely, J. Franks determined  necessary and  sufficient
conditions on an abstract Lyapunov graph
  to be associated with an NS flow on $S^3$.
 Following the idea of J. Franks, K. de Rezende and her cooperators
      classified Smale flows on $S^3$ (\cite{Re1}), gradient-like flows on 3-manifolds (\cite{Re2})
      and smooth flows on surfaces (\cite{RF}).

      NS flows were first introduced by J. Franks in the 1980s, see
\cite{F1}, \cite{F2}, and \cite{F5}. An NS flow is a structurally
stable flow with one-dimensional invariant sets and without
singularities. More restricted Smale flows, namely nonsingular
Morse-Smale flows (abbreviated as NMS flows),  have been effectively
studied by associating NMS flows with a kind of combinatorial tool,
i.e. round handle decomposition, see \cite{As}, \cite{Mo} and
\cite{Wa}. However, for general NS flows, the situation becomes more
complicated. M. Sullivan \cite{Su} and the author \cite{Yu1}
       used template and knot theory to describe more embedding
       information of a special type of nonsingular Smale flows on
       $S^{3}$. J. Franks (\cite{F1},  \cite{F2}) used homology to describe some
       embedding information of NS flows on
       3-manifolds.

More recently, F. B\'eguin and C. Bonatti \cite{BB} gave several new
concepts and results to describe the behavior of a neighborhood of a
basic set of a Smale flow on a 3-manifold.  Their paper is an
important step in the classification, up to topological equivalence,
of NS flows on 3-manifolds. It provided useful canonical
neighborhoods of one-dimensional hyperbolic basic sets on
3-manifolds, i.e. filtrating neighborhoods. We note that the same
concept was also discovered by J. Franks in \cite{F5}, who named it
``building block".

The work of F. B\'eguin and C. Bonatti \cite{BB} indicates that  the
topological structures of all Lyapunov graphs  of a Smale flow on a
3-manifold are the same and a filtrating neighborhood is always
associated with a vertex in a Lyapunov graph (More details could be
found in Section \ref{section6}). These facts convinces us that a
generalization of J. Franks' work for all 3-manifolds is an
efficient way of understanding NS flows on 3-manifolds. In this
direction, N. Oka has made some progress in \cite{Ok}. He
generalized J. Franks' work to all lens
 spaces with odd order fundamental group. The results
 and the techniques in this paper  are similar to J. Franks'.
 For example, J. Franks  proved that a Lyapunov graph of any NS flow on
$S^{3}$ must be a tree and the weight of any edge of $L$ is 1, i.e.,
 the regular level sets of any NS flow on $S^{3}$ must be tori. The same phenomena
 occur
 in the study of NS flows on a lens space with odd order fundamental
 group.

In this paper,  we mainly study Lyapunov graphs of NS flows  on
$S^{1}\times S^{2}$. The reasons why we choose $S^{1}\times S^{2}$
are:
\begin{itemize}
\item a Lyapunov graph of an NS flow on $S^{1}\times S^{2}$ may not be a
tree;
\item the weight of an edge of a Lyapunov graph of an NS flow may not be 1;
\item for our purpose, the topology of $S^{1}\times S^{2}$ is well understood.
\end{itemize}
We use three theorems, i.e., Theorem \ref{theorem5.1}, Theorem
\ref{theorem5.2} and Theorem \ref{theorem5.5} to determine necessary
and sufficient conditions on an abstract Lyapunov graph to be
associated with an NS flow on $S^1 \times S^2$. The fact that we
need three theorems is due to some topological information of $S^1
\times S^2$ (see Proposition \ref{proposition4.3}).

In addition, to generalize  J. Franks' work for all 3-manifolds,  we
find that it is important to consider a kind of vertices, i.e.,
singular vertices (the definition can be found in Section
\ref{section6}). In the end of this paper, we
 study singular vertices of NS flows on irreducible
3-manifolds. The main result (Proposition \ref{proposition6.3}) is
that the number of the  singular vertices in a  Lyapunov graph
associated with an NS flow  on an irreducible, closed orientable
3-manifold is restricted by  the Haken number of the 3-manifold.

This paper is organized as follows. In Section \ref{section2}, we
give some definitions and detailed background
     knowledge for Lyapunov graphs of NS flows on 3-manifolds. In Section \ref{section3}, we discuss some
     connections between Smale flows and homology. In Section \ref{section4}, we prove some general properties
     about NS flows on $S^1 \times S^2$. In Section \ref{section5}, we state and
     prove the main theorems of this paper. In Section \ref{section6} we obtain some results on singular vertices of Lyapunov
     graphs of NS flows on 3-manifolds.

\section{Lyapunov graph of nonsingular Smale flow} \label{section2}

 A smooth flow $\phi_{t}$ on
a compact manifold $M$ is called a \textit{Smale flow} if:
\begin{enumerate}

\item the chain recurrent set $R(\phi_{t})$ has hyperbolic structure;

\item  $\dim (R(\phi_{t}))\leq 1$;

\item $\phi_{t}$ satisfies the transversality condition and noncyclic condition.
\end{enumerate}
If a Smale flow $\phi_{t}$ has no singularities, we call
  $\phi_{t}$  a \textit{nonsingular Smale flow} (abbreviated as an \textit{NS flow}).
   In particular,  if $R(\phi_{t})$ consists entirely of
closed orbits, $\phi_{t}$ is called a \textit{nonsingular
Morse-Smale flow} (abbreviated as an \textit{NMS flow}). For a Smale
flow $\phi_t$,
   if $R(\phi_{t})$ consists entirely of singularities,
  $\phi_{t}$ is called a \textit{gradient-like flow}.

A diffeomorphism $f$ on a closed manifold $M$ is called a
\emph{Morse-Smale} diffeomorphism if:
\begin{enumerate}
\item the chain recurrent set $R(f)$ has hyperbolic structure and consists entirely of singularities;
\item for any $x, y \in R(f)$, the stable manifold $W^s (x)$ is transverse
to the unstable manifold $W^u (y)$.
\end{enumerate}
A Morse-Smale diffeomorphism $f$ is called a \emph{gradient-like}
diffeomorphism if for any $x, y \in R(f)$, $W^s (x)\cap W^u (y)\neq
\emptyset$ forces that $\dim W^s (x)> \dim W^s (y)$.

\begin{definition}
 An \emph{abstract Lyapunov graph} is a
finite, connected and oriented graph $L$ satisfying the following
two conditions:

(1) $L$ possesses no oriented cycles;

(2) each vertex of $L$ is labeled with a chain recurrent flow on a
compact space.
\end{definition}

\begin{theorem} \textbf{(R. Bowen \cite{Bo})}

If $\phi_{t}$ is a flow with hyperbolic chain recurrent set and
$\Lambda$ is a 1-dimensional basic set, then $\phi_{t}$ restricted
to $\Lambda$ is topologically equivalent to the suspension of a
basic subshift of finite type (i.e., a subshift associated with an
irreducible matrix (abbreviated as SSFT)).
\end{theorem}

This theorem enables us to label a vertex of the (abstract) Lyapunov
graph that represents the 1-dimensional basic sets of a Smale flow
with the suspension of an SSFT $\sigma(A)$. For simplicity we will
label the vertex with the nonnegative integer irreducible matrix
$A$. For a vertex labeled with matrix $A=(a_{ij})$, let
$B=(b_{ij})$, where $b_{ij}\in \mathbb{Z}/2$ and $b_{ij}\equiv
a_{ij} (mod 2)$ and $k=\dim \ker((I-B):F_{2}^{m}\rightarrow
F_{2}^{m})$, $F_{2}=\mathbb{Z}/2$. The number of incoming (outgoing)
edges is denoted by $e^{+}$ ($e^{-}$). Denote the weight of an edge
of a Lyapunov graph by the genus of the regular level set of the
edge. Let $g_{j}^{+}$ ($g_{j}^{-}$) be the weight on an incoming
(outgoing) edge of the vertex. In this paper, we always denote the
initial vertex and the terminal vertex of an oriented edge $E$ by
$i(E)$ and $t(E)$ respectively. Whenever we talk about a Lyapunov
graph $L$ of an NS flow on a 3-manifold $M$, L is always associated
with a Lyapunov function $g:M\rightarrow R$ and a map
$h:M\rightarrow L$ such that $g=\pi \circ h$. Here $\pi:
L\rightarrow R$ is the natural projection.

The following classification theorem is due to J. Franks \cite{F5}:

\begin{theorem}\label{theorem2.4}
Let $L$ be an abstract Lyapunov graph $L$. $L$ is associated with an
NS flow $\phi_{t}$ on $S^{3}$ if and only if the following
conditions hold.
\begin{enumerate}

\item The underlying graph $L$ is a tree with exactly one edge attached
to each closed orbit sink or closed orbit source vertex. The weight
of any edge of $L$ is 1.

\item If a vertex is labeled with an SSFT with matrix $A_{m\times
m}$, then
\begin{equation}\label{1}
 \begin{split}
    &0< e^{+}\leq k+1, 0< e^{-}\leq k+1\text{\and}\\
    &k+1\leq e^{+}+ e^{-}.
 \end{split}\nonumber
\end{equation}
\end{enumerate}
\end{theorem}

K. Rezende \cite{Re1} generalized J. Franks' theorem to Smale flows
on $S^3$:

\begin{theorem}\label{theorem2.41}
Let $L$ be an abstract Lyapunov graph. $L$ is  associated with a
Smale flow $\phi_{t}$ on $S^{3}$ if and only if the following
conditions hold.

(1) The underlying graph $L$ is a tree with exactly one edge
attached to each sink or source vertex.

(2) If a vertex is labeled with an SSFT with matrix $A_{m\times m}$,
then we have
\begin{equation}\label{1}
 \begin{split}
    &e^{+}> 0, e^{-}> 0,\\
    &k+1-\sum_{i=1}^{e^{-}}g_{i}^{-}\leq e^{+}\leq k+1  \text{\ and}\\
&k+1-\sum_{j=1}^{e^{+}}g_{j}^{+}\leq e^{-}\leq k+1.
 \end{split}\nonumber
\end{equation}

(3) All vertices must satisfy the Poinc\'are-Hopf condition. Namely,
if a vertex is labeled with a singularity of index r, then
\begin{equation}\label{1}
(-1)^{r}=e^{+}-e^{-}-\sum g_{j}^{+}+ \sum g_{i}^{-}, \nonumber
\end{equation}

and if a vertex is labeled with a suspension of an SSFT or a
periodic orbit, then
\begin{equation}\label{1}
0=e^{+}-e^{-}-\sum g_{j}^{+}+ \sum g_{i}^{-}.
 \nonumber
\end{equation}
\end{theorem}

The following theorem is due to \cite{CR}:
\begin{theorem}\label{theorem2.5}
Suppose a closed orientable 3-manifold M admits
      an NS flow $\phi_{t}$ with Lyapunov graph $L$, then $\beta_{1}(L)\leq \beta_{1}(M)$.
      Here $\beta_{1}(L)$ and $\beta_{1}(M)$ are the first Betti
      numbers of $L$ and $M$ respectively.
\end{theorem}

Throughout this paper, if $V$ is a manifold, $mV$ means the
connected sum of $m$ copies of $V$. The following theorem
(\cite{Yu2}) tells us some relation between the topology of
3-manifolds and the weights of edges of Lyapunov graphs of NS flows:
\begin{theorem}\label{theorem2.6}
Suppose a closed orientable 3-manifold M admits
      an NS flow $\phi_{t}$ with Lyapunov graph $L$. Then M admits an NS flow which
      has a regular level set homeomorphic to
 $(n+1)T^{2}$  $(n\in \mathbb{Z}, n\geq 0)$
      if and only if $M=M'\sharp n S^{1}\times S^{2}$. Here
      $M'$ is any closed orientable 3-manifold. Furthermore, $\beta_{1}(L)\geq n$.
\end{theorem}

\section{Homology and Smale flows}\label{section3}
\begin{definition}
If $\phi_t: M\rightarrow M$ is a flow with hyperbolic chain
recurrent set and its basic sets are $\{\Lambda_{i}\}$ ($i=1,...,
n$), then  a \emph{filtration} associated with $\phi_t$ is a
collection of submanifolds $M_0 \subset M_1 \subset ... \subset M_n
=M$ such that
\begin{enumerate}
\item $\phi_{t}(M_i) \subset \interior M_i$, for any $t>0$;
\item $\Lambda_i = \cap_{t=-\infty}^{\infty} \phi_{t}(M_i -
M_{i-1})$.
\end{enumerate}
\end{definition}

The following theorem is due to Bowen and J. Franks \cite{F2}:

\begin{theorem}\label{theorem3.2}
Let $\phi_t$ be a Smale flow and $M_{i}, i=1,...,n$ be a filtration
associated with $\phi_t$. Suppose $\Lambda_i
=\cap_{t=-\infty}^{\infty} \phi_{t}(M_i - M_{i-1})$ is an index $u$
basic set labeled with a matrix $A_{n\times n}$, then:
\begin{equation}\label{2}
 \begin{split}
    &H_{k}(M_{i},M_{i-1},F_2)\cong 0, \text{\ if}~  k\neq u, u+1;\\
    &H_{u}(M_{i},M_{i-1},F_2)\cong F_{2}^{n}/(I-B)F_{2}^{n};\\
    &H_{u+1}(M_{i},M_{i-1},F_2)\cong \ker ((I-B)~ on~ F_{2}^{n}).
 \end{split}\nonumber
\end{equation}
\end{theorem}

Let $M$ be a closed orientable 3-manifold, $\phi_t$ be a Smale flow
on $M$ and $g$ be a Lyapunov function associated with $\phi_t$.
Assume that $c\in \mathbb{R}$ is a singular value and $g^{-1}(c)$ is
associated to a basic set labeled with a matrix $B_{n\times n}$. Set
$X=g^{-1}((-\infty, c+\epsilon])$, $Y=g^{-1}([c+\epsilon,+\infty))$
and $Z=g^{-1}((-\infty,c-\epsilon])$. Throughout this paper $k=\dim
\ker(I-B)$. The homology coefficients in this paper shall be taken
in $F_2 =\mathbb{Z} /2$.

\begin{proposition}\label{proposition3.3}

\begin{enumerate}
\item $e^+ \leq k+1+\dim H_{2}(Z) + \dim H_{2}(Y)$;
\item $e^- \leq k+1+\dim H_{2}(Z) + \dim H_{2}(Y)$;
\item $k \leq e^+ -1+\dim H_{1}(M)+ \dim H_{1}(Z)-\dim
    H_{2}(Z)$;
\item $k \leq e^- -1+\dim H_{1}(M)+ \dim H_{1}(Y)-\dim
    H_{2}(Y)$.
\end{enumerate}
\end{proposition}

\begin{proof}
We consider the Mayer-Vietoris exact sequence:
\begin{equation}\label{3}
\begin{split}
&H_{3}(X)\oplus H_{3}(Y)\rightarrow H_{3}(X\cup Y)\rightarrow
H_{2}(X\cap Y)\xrightarrow{d_*} H_{2}(X)\oplus
H_{2}(Y)\\
&\xrightarrow{c_*} H_{2}(X\cup Y)
\end{split}
\end{equation}
and the exact sequence for $(X,Z)$ :
\begin{equation}\label{4}
H_{3}(X,Z)\rightarrow H_{2}(Z)\rightarrow H_{2}(X)\xrightarrow{b_*}
H_{2}(X,Z)\rightarrow H_{1}(Z)\rightarrow H_{1}(X).
\end{equation}

Both $X$ and $Y$ are compact three manifolds with boundary, so
$H_{3}(X)\cong H_{3}(Y)\cong 0$. $X\cup Y = M$, so $H_{3}(X\cup Y)
\cong F_2$. $X\cap Y$ is composed of $e^+$ closed orientable
surfaces, so $H_{2}(X\cap Y)\cong e^{+} F_2$.  Since (\ref{3}) is an
exact sequence, $\dim \imag d_* =\dim \ker c_*\leq \dim H_{2}(X)+
\dim H_{2}(Y)$. It implies that $\dim \imag d_* =e^+ -1$. Therefore,
we have $e^+ -1\leq \dim H_{2}(X)+ \dim H_{2}(Y)$. By the exact
sequence (\ref{4}), we have $\dim H_{2}(X) \leq \dim H_{2}(Z) + \dim
H_{2}(X,Z)$. By Theorem \ref{theorem3.2}, we have $\dim H_{2}(X,Z)=
\dim \ker(I-B)=k$. Therefore, (1) in Proposition
\ref{proposition3.3} is proved. (2) in Proposition
\ref{proposition3.3} can be proved similarly.

By Theorem \ref{theorem3.2}, we have $H_{3}(X,Z)=0$. Therefore, by
the exact sequence (\ref{4}),
\begin{equation}\label{5}
k=\dim H_{2}(X,Z)\leq \dim H_{2}(X)+\dim H_{1}(Z)-\dim H_{2}(Z).
\end{equation}
 Now we consider the reduced exact sequence for $(M,Y)$:
\begin{equation}\label{6}
\widetilde{H_{1}}(Y)\rightarrow \widetilde{H_{1}}(M)\rightarrow
\widetilde{H_{1}}(M,Y)\rightarrow \widetilde{H_{0}}(Y) \rightarrow
\widetilde{H_{0}}(M).
\end{equation}
Since $\widetilde{H_{0}}(M)=0$ and  $\dim \widetilde{H_{0}}(Y)=e^+
-1$, by the exact sequence (\ref{6}), we have $\dim H_{1}(M,Y)\leq
e^+ -1+ \dim H_{1}(M)$. On the other hand, by Lefschetz duality,
$H_{1}(M,Y)\cong H_{2}(M-Y,M-M)=H_{2}(X)$. Therefore, $\dim
H_{2}(X)\leq e^+ -1+ \dim H_{1}(M)$. By (\ref{5}), $k\leq e^+ -1+
\dim H_{1}(M)+\dim H_{1}(Z)-\dim H_{2}(Z)$. Therefore, (3) in
Proposition \ref{proposition3.3} is proved. (4) in Proposition
\ref{proposition3.3} can be proved similarly.
\end{proof}

\begin{lemma}\label{lemma3.4}
Let $K$ be a knot in $S^3$ and $K^c$ be the complement space of $K$
in $S^3$, then $H_{1}(K^c)\cong F_2$ and $H_{2}(K^c)\cong 0$.
\end{lemma}

\begin{proof}
By Lefschez duality theorem, we have $H_{1}(K^c)\cong H_{2}(S^3,K)$
and $H_{2}(K^c)\cong H_{1}(S^3,K)$. By the exact sequence:
$H_{2}(S^3)\rightarrow H_{2}(S^3,K)\rightarrow H_{1}(K)\rightarrow
H_{1}(S^3)\rightarrow H_{1}(S^3,K)$, it is easy to show that
$H_{2}(S^3,K)\cong F_2$ and $H_{1}(S^3,K)\cong 0$. Therefore,
$H_{1}(K^c)\cong F_2$ and $H_{2}(K^c)\cong 0$.
\end{proof}

\begin{corollary}\label{corollary3.5}
If each component of $Y$ and $Z$ is homeomorphic to a knot
complement, then $e^{+}\leq k+1$, $e^{-}\leq k+1$ and $k+1\leq
e^{+}+e^{-}+\dim H_{1}(M)$.
\end{corollary}
\begin{proof}
It follows easily from Proposition \ref{proposition3.3} and Lemma
\ref{lemma3.4}.
\end{proof}

\section{Some Properties on Lyapunov graphs of NS flows on $S^1 \times S^2$}\label{section4}

\begin{proposition}\label{proposition4.1}
Let $i:T^2 \hookrightarrow S^1 \times S^2$ be an embedding map and
$i_{*}: \pi_{1}(T^2)\rightarrow \pi_{1}(S^1 \times S^2)$ be the
homomorphism induced by $i$. One and only one of the following three
possibilities occurs.
\begin{enumerate}
\item If $i(T^2)$ is inseparable in $S^1 \times S^2$, then $S^1 \times S^2- i(T^2)\cong (O\sqcup
K)^{c}$. Here $(O\sqcup K)^{c}$ is the compliment of a
two-component-link in $S^3$, where $O$ is a trivial knot and $K$ can
be any knot unlinked with $O$.

\item If $i(T^2)$ is separable in $S^1 \times S^2$ and $\imag
i_{*}=0$, then $$S^1 \times S^2- i(T^2)\cong K^{c} \sqcup (S^1
\times D^2)\sharp (S^1 \times D^2)$$ or $$S^1 \times S^2-
i(T^2)\cong (S^1 \times D^2) \sqcup (S^1 \times D^2)\sharp K^{c}.$$
Here $S^1 \times D^2$ is an open solid torus and $K$ can be any
knot.

\item If $i(T^2)$ is separable in $S^1 \times S^2$ and $\imag
i_{*}\neq 0$, then $i(T^2)$ bounds a solid torus in $S^1 \times
S^2$.
\end{enumerate}
\end{proposition}

\begin{figure}[htp]
\begin{center}
  \includegraphics[totalheight=3.1cm]{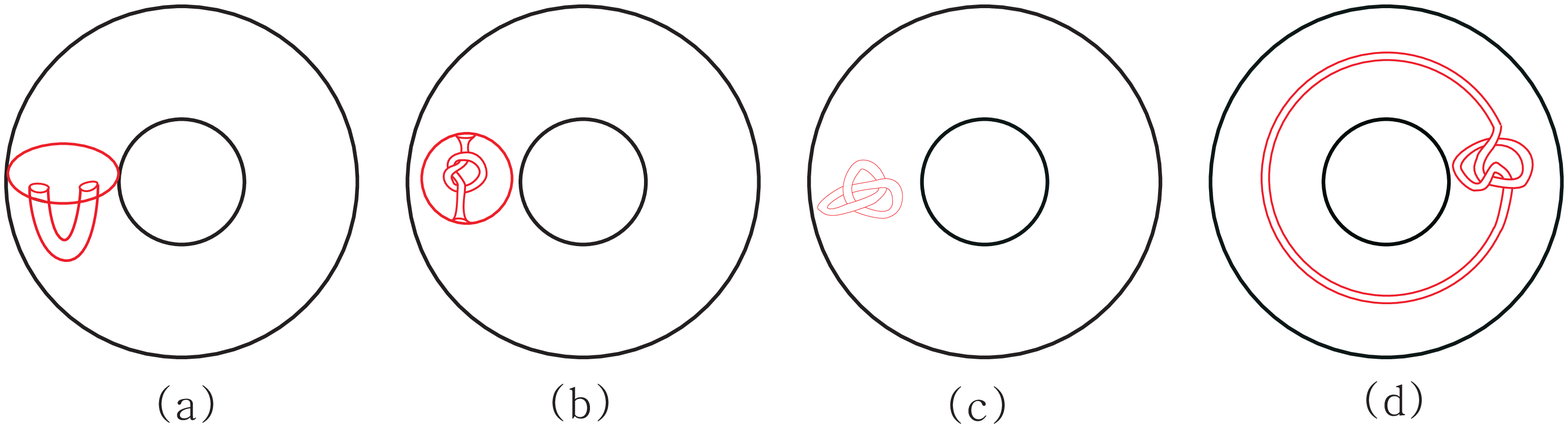}\\
  \caption{}\label{figure1}
  \end{center}
\end{figure}

\begin{proof}
Since $\pi_1 (T^2)\cong 2\mathbb{Z}$ and $\pi_1 (S^1 \times
S^2)\cong \mathbb{Z}$, $\ker i_* \neq0$.  Therefore,  by Dehn's
lemma (see [He]), there exists a nontrivial simple closed curve
$c\subset i(T^2)$ such that $c$ bounds a disk $D$ in $S^1 \times
S^2- i(T^2)$. Cutting $i(T^2)$ along a ribbon neighborhood $N(c)$ of
$c$ and pasting two copies of $D$ to $\partial (i(T^2)-N(c))$, we
obtain a sphere $S_{0}^{2}$. One and only one of the following three
possibilities occurs.
\begin{enumerate}
\item If $i(T^2)$ is inseparable in $S^1 \times S^2$, then $S_{0}^{2}$
is also inseparable in $S^1 \times S^2$. Therefore, $S_{0}^{2}$ is
isotopic to $\{pt\}\times S^2$ where $pt$ is a point in $S^1$. So it
is easy to show that $S^1 \times S^2- i(T^2)\cong (O\sqcup K)^{c}$.
This case is illustrated by  Figure \ref{figure1} (a).
\item If $i(T^2)$ is separable in $S^1 \times S^2$ and $\imag
i_{*}=0$, then $S_{0}^{2}$ bounds a 3-ball in $S^1 \times S^2$.
Moreover,
\begin{enumerate}
\item if $c$ is in the 3-ball, then $S^1 \times S^2- i(T^2)\cong
K^{c} \sqcup (S^1 \times D^2)\sharp (S^1 \times D^2)$ (see Figure
\ref{figure1} (b));
\item if $c$ is not in the 3-ball, then $S^1 \times S^2- i(T^2)\cong (S^1 \times D^2) \sqcup (S^1 \times D^2)\sharp K^{c}$ (see Figure
\ref{figure1} (c)).
\end{enumerate}
\item If $i(T^2)$ is separable in $S^1 \times S^2$ and $\imag
i_{*}\neq 0$, then $S_{0}^{2}$ bounds a 3-ball in $S^1 \times S^2$
and $c$ is not in the 3-ball. Therefore, $i(T^2)$ bounds a solid
torus in $S^1 \times S^2$, as shown in Figure \ref{figure1} (d).
\end{enumerate}
\end{proof}

From the proposition above, it is easy to show the following
corollary.

\begin{corollary}\label{corollary4.2}
Let $i:T^2 \hookrightarrow S^1 \times S^2$ be an embedding map such
that $i(T^2)$ is separable in $S^1 \times S^2$, then $i(T^2)$ bounds
a knot complement in $S^1 \times S^2-i(T^2)$.
\end{corollary}

The following lemma is Theorem 4.7 in \cite{Re1}. It is proved by
using Poincar\'e-Hopf formula.

\begin{lemma}\label{lemma4.3}
Suppose $\phi_{t}$ is a smooth flow on an odd-dimensional manifold
$M$ which transverses outside to $\partial M^{-}$ and transverses
inside to $\partial M^{+}$, where $\partial M=\partial M^{+} \cup
\partial M^{-}$. Then $\sum I_{p}=\frac{1}{2}(X(\partial
M^{+})-X(\partial M^{-}))$ where the summation is taken over all
singularities in $M$, $I_{p}$ is the index of the singularity $p$
and $X$ denotes the Euler characteristic.
\end{lemma}

\begin{proposition}\label{proposition4.3}
Let $\phi_t$ be an NS flow on $S^{1}\times S^{2}$ with Lyapunov
graph $L$.
\begin{enumerate}
\item If each regular level set of $(\phi_t,L)$ is separable, then each regular level set of $(\phi_t,L)$ is homeomorphic to a torus and
$L$ is a tree.
\item If at least one of the regular level sets of $(\phi_t,L)$ is inseparable,
 then there are exactly the following two possibilities.
\begin{enumerate}
\item If $\Sigma$ is homeomorphic to $T^2$, then each regular level
set of $(\phi_t,L)$ is homeomorphic to $T^2$. Moreover,
$\beta_{1}(L)=1$.
\item If $\Sigma$ is not homeomorphic to $T^2$, then
$\Sigma$ is homeomorphic to $S^2$ or $2T^2$ and $\beta_{1}(L)=1$.
Furthermore, there exist at least one regular level set which is
homeomorphic to $S^2$ and at least one regular level set which is
homeomorphic to $2T^2$.
\end{enumerate}
\end{enumerate}
\end{proposition}

\begin{proof}
Let $\Sigma$ be a regular level set of $(\phi_t,L)$.

Suppose each regular level set of $(\phi_t,L)$ is separable. For a
regular level set $\Sigma$, $S^1 \times S^2 = M_{1}\cup_{\Sigma}
M_{2}$ where $M_1$ and $M_2$ are two compact 3-manifolds with
boundaries. By Lemma \ref{lemma4.3}, we have $X(\Sigma)=0$.
Therefore, $\Sigma$ is homeomorphic to a torus. If $L$ isn't a tree,
then there exists a cycle in $L$. The inverse image of a regular
point in the cycle is a inseparable regular level set in $S^1 \times
S^2$. It contradicts our assumption. Therefore, (1) in Proposition
\ref{proposition4.3} is proved.

 Now we suppose there exists a regular level set $\Sigma$ of $(\phi_t,L)$ which is
inseparable.

If $\Sigma \cong T^2$, by Proposition \ref{proposition4.1}, $N=S^1
\times S^2 - \Sigma \cong (O\cup K)^{c}$ where $O$ is a trivial knot
and $K$ is an arbitrary knot unlinked with $O$. Attaching a standard
solid torus neighborhood of a closed orbit attractor (round
0-handle) and a standard solid torus neighborhood of a closed orbit
repeller (round 2-handle) to $(N, \phi_t |N)$ suitably, we obtain an
NS flow $\psi_t$ on $S^3$. By Theorem \ref{theorem2.4}, each regular
level set of $\psi_t$ is homeomorphic to a torus. Therefore, each
regular level set of $\phi_t$ is homeomorphic to $T^2$. Since there
exists an inseparable regular level set,  it is obvious that
$\beta_{1}(L)\geq 1$. On the other hand, by Theorem
\ref{theorem2.5}, $\beta_{1}(L)\leq 1$. Therefore, $\beta_{1}(L)=1$.
(a) of (2) in Proposition \ref{proposition4.3} is proved.

If $\Sigma$ is not homeomorphic to a torus,  by a similar argument,
we can show that $\beta_{1}(L)=1$. By Theorem \ref{theorem2.6}, we
have $\Sigma \cong S^2$ or $\Sigma \cong 2T^2$. Let $g:M\rightarrow
R$ be a Lyapunov function associated with $L$ and $x\in \mathbb{R}$
be a regular value. Suppose $h:M\rightarrow L$ is a map such that
$g=\pi \circ h$. Here $\pi: L\rightarrow R$ is the natural
projection. Suppose $g^{-1}(x)=\Sigma_1 \sqcup \Sigma_2 \sqcup ...
\sqcup \Sigma_s$. By Lemma \ref{lemma4.3}, $\sum_{i=1}^{s}
X(\Sigma_i)=0$. Since for any $i\in \{1,...,s\}$, $\Sigma_i$ is
homeomorphic to one of $S^2$, $T^2$ and  $2T^2$, it is easy to show
that there exist at least one regular level set which is
homeomorphic to $S^2$ and at least one regular level set which is
homeomorphic to $2T^2$. Hence (b) of (2) in Proposition
\ref{proposition4.3} is proved.
\end{proof}

\begin{remark}
The result in Proposition \ref{proposition4.3} (1) is still true if
we generalize $S^1 \times S^2$ to any other closed 3-manifold $M$.
The proof is similar. Indeed, if we change $S^1 \times S^2$ to any
other closed 3-manifold $M=M' \sharp n S^{1} \times S^{2}$ where
$M'$ is prime to $S^1 \times S^2$, by Theorem \ref{theorem2.6}, each
regular level set of $(\phi_t,L)$ is homeomorphic to one of
$S^2,T^2,...,(n+1)T^2$.
\end{remark}

\begin{proposition} \label{proposition4.5}
Let $\phi_t$ be an NS flow on $S^{1}\times S^{2}$ with Lyapunov
graph $L$ such that each regular level set of $\phi_t$ is
homeomorphic to a torus.
\begin{enumerate}
\item If there exists an inseparable regular level set, then
for any basic set of $\phi_t$, $e^+ \leq k+1$, $e^- \leq k+1$ and
$k+1\leq e^+ + e^-$.
\item  If each regular level set of $\phi_t$ is separable, then
for any basic set of $\phi_t$, $e^+ \leq k+1$, $e^- \leq k+1$ and
$k\leq e^+ + e^-$. Furthermore, there exists at most one basic set
of $\phi_t$ which satisfies that $k= e^+ + e^-$.
\end{enumerate}
\end{proposition}

\begin{proof}
If there exists an inseparable regular level set $\Sigma$ in $S^1
\times S^2$, by Proposition \ref{proposition4.1}, $N=S^1 \times S^2
- \Sigma \cong (O\cup K)^{c}$ where $O$ is a trivial knot and $K$ is
a knot unlinked with $O$. Attaching a round 0-handle  and a round
2-handle to $(N, \phi_t |N)$ suitably, we obtain an NS flow $\psi_t$
on $S^3$. By Theorem \ref{theorem2.4}, each basic set of $\psi_t$
satisfies that $e^+ \leq k+1$, $e^- \leq k+1$ and $k+1\leq e^+ +
e^-$. Therefore, for any basic set of $\phi_t$, $e^+ \leq k+1$, $e^-
\leq k+1$ and $k+1\leq e^+ + e^-$.

If each regular level set of $\phi_t$ is separable, by Proposition
\ref{proposition4.3}, each regular level set is homeomorphic to a
torus and $L$ is a tree. Let $\Sigma$ be a regular level set of
$\phi_t$, by Corollary \ref{corollary4.2}, $\Sigma$ bounds a knot
complement in $S^1 \times S^2 - \Sigma$. Let $g$ be a Lyapunov
function associated with $L$. Assume that $c\in \mathbb{R}$ is a
singular value and $g^{-1}(c)$ is associated with a basic set
$\Lambda$ labeled with a matrix $B_{n\times n}$. Set
$X=g^{-1}((-\infty, c+\epsilon])$, $Y=g^{-1}([c+\epsilon,+\infty))$
and $Z=g^{-1}((-\infty,c-\epsilon])$. If each component of $Y$ and
$Z$ is a knot complement, by Corollary \ref{corollary3.5}, $e^+ \leq
k+1$, $e^- \leq k+1$ and $k\leq e^+ + e^-$. Otherwise, without loss
of generality, we suppose that a component of $Y$, denoted by $Y_1$,
is not a knot complement. Therefore, $S^1 \times S^2 - Y_1$ is a
knot complement. By an argument similar to the proof of (1) in
Proposition \ref{proposition4.5}, we have $e^+ \leq k+1$, $e^- \leq
k+1$ and $k+1\leq e^+ + e^-$. In summary, $e^+ \leq k+1$, $e^- \leq
k+1$ and $k\leq e^+ + e^-$.  We choose a regular level set for each
edge of $L$ and these regular level sets divide $S^1 \times S^2$
into m manifolds with boundary, $N_1, ..., N_m$. Each $N_i$
corresponds to a basic set $\Lambda_i$. Since each regular level set
of $\phi_t$ bounds a knot complement, at most one of $N_1,...,N_m$
is not a link complement. If there is one, we assume without loss of
generality that it is $N_1$. For an $(N_j,\Lambda_j)$ ($j\neq 1$),
by an argument similar to the proof of (1) in Proposition
\ref{proposition4.5}, $k+1\leq e^+ + e^-$. Therefore, there exists
at most one basic set of $\phi_t$ which satisfies that $k= e^+ +
e^-$.
\end{proof}

\section{Main results: Lyapunov graphs of NS flows on $S^{1}\times S^{2}$}\label{section5}
In this section, we prove Theorem \ref{theorem5.1}, Theorem
\ref{theorem5.2} and Theorem \ref{theorem5.5}, which give necessary
and sufficient conditions on an abstract Lyapunov graph  to be
associated with an NS flow on $S^1 \times S^2$. The fact that we
need three theorems is due to Proposition \ref{proposition4.3}.

\begin{theorem}\label{theorem5.1}
Let $L$ be an abstract Lyapunov graph, then $L$ is associated with
an NS flow $\phi_{t}$ on $S^{1}\times S^{2}$ such that there exists
an inseparable regular level set which is homeomorphic to $T^{2}$ if
and only if the following conditions hold.
\begin{enumerate}

\item The underlying graph $L$ is an oriented graph and $\beta_{1}(L)=1$
with exactly one edge attached to each closed orbit sink or closed
orbit source vertex.  The weight of any edge of $L$ is 1.

\item If a vertex is labeled with an SSFT with matrix $A_{m\times
m}$, then
\begin{equation}\label{1}
 \begin{split}
    &0< e^{+}\leq k+1, 0< e^{-}\leq k+1 \text{\ and}\\
    &k+1\leq e^{+}+ e^{-}.
 \end{split}\nonumber
\end{equation}
\end{enumerate}
\end{theorem}

\begin{proof}
\emph{The necessity}. Let $\phi_t$ be an NS flow on $S^1 \times S^2$
which satisfies that there exists an inseparable regular level set
which is homeomorphic to $T^{2}$. Suppose $L$ is a Lyapunov graph of
$\phi_t$. Since there exists an inseparable regular level set which
is homeomorphic to $T^{2}$, by (a) of (2) in Proposition
\ref{proposition4.3}, we have $\beta_{1}(L)=1$ and the weight of any
edge of $L$ is 1. By (1) in Proposition \ref{proposition4.5}, we
have $0< e^+ \leq k+1$, $0< e^- \leq k+1$ and $k+1\leq e^+ + e^-$.
The noncyclic condition of NS flow implies that $L$ possesses no
oriented cycles.

\begin{figure}[htp]
\begin{center}
  \includegraphics[totalheight=4cm]{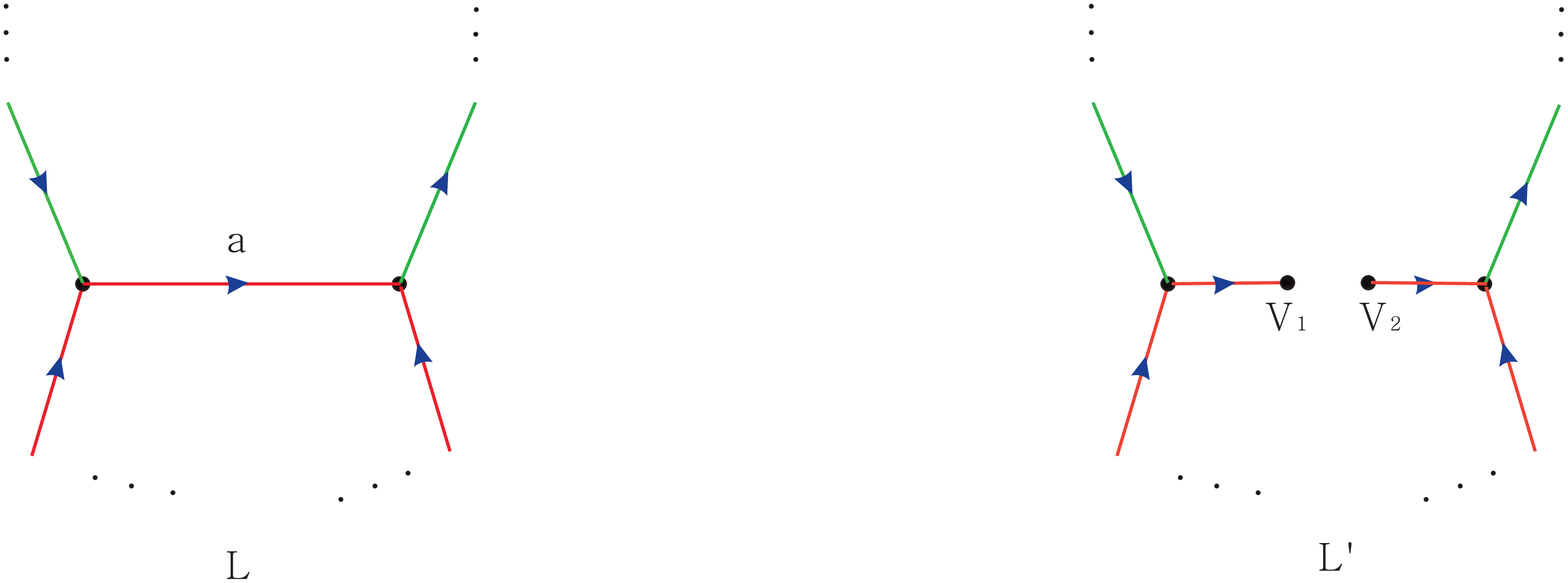}\\
  \caption{}\label{figure2}
  \end{center}
\end{figure}

\emph{The sufficiency}. Let $L$ be a Lyapunov graph which satisfies
(1) and (2) of Theorem \ref{theorem5.1}. Suppose $a$ is an edge in
$L$ such that $a$ belongs to the cycle of $L$. We can obtain a new
Lyapunov graph $L'$ by cutting $L$ along $a$, then adding two
vertices $V_1$ and $V_2$ which are labeled by a closed orbit sink
and a closed orbit source respectively, as shown in Firgure
\ref{figure2}. It is easy to show that $L'$ satisfies the sufficient
condition of Theorem \ref{theorem2.4}. Therefore, there exists an NS
flow $\psi_t$ on $S^3$ with Lyapunov graph $L'$. Suppose the sink
labeled with $V_1$ and the source labeled with $V_2$ in $\psi_t$ are
$K_1$ and $K_2$ respectively. By the constructions of NS flows in
\cite{F5}, $K_1 \sqcup K_2$ is a two-component unlinked link such
that $K_1$ and $K_2$ both are trivial knots. Cutting two standard
solid torus neighborhoods $N(K_1)$ and $N(K_2)$ of $K_1$ and $K_2$
respectively, then pasting $S^3-(N(K_1)\sqcup N(K_2))$ along
$\partial N(K_1)$ and $\partial N(K_2)$ suitably, we obtain an NS
flow $\phi_t$ on $S^1 \times S^2$ with Lyapunov graph $L$.
\end{proof}

\begin{theorem}\label{theorem5.2}
Let $L$ be an abstract Lyapunov graph, then $L$ is associated with
an NS flow $\phi_{t}$ on $S^{1}\times S^{2}$ such that there exists
an inseparable regular level set which is homeomorphic to $S^{2}$ if
and only if the following conditions hold.
\begin{enumerate}

\item The underlying graph $L$ is an oriented graph and $\beta_{1}(L)=1$
 with exactly one edge attached to each closed orbit sink
or closed orbit source vertex. Suppose the cycle $C$ of $L$ is
composed of $n$ edges. Then $C$ and $L$ satisfy the following
conditions.
\begin{enumerate}
\item The weight of an edge of $C$ is 0 or 2. Moreover, there exist at least
 one  weight 0 edge and one weight 2 edge in $C$.
\item The orientation of a weight 0 edge in $C$ is opposite to the
orientation of a weight 2 edge in $C$.
\item The weight of any other edge of $L$ is 1.
\end{enumerate}

\item If a vertex is labeled with an SSFT with matrix $A_{m\times
m}$, then
\begin{equation}\label{1}
    0< e^{+}\leq k+1, 0< e^{-}\leq k+1.
\nonumber
\end{equation}
Moreover,
\begin{enumerate}
\item if there exist a weight 0 edge starting at the vertex and a weight 0 edge terminating at the vertex, then
$k+2\leq e^+ +e^-$;
\item if there exist a weight 2 edge starting at the vertex and a weight 2 edge terminating at the vertex, then
$k\leq e^+ +e^-$;
\item otherwise, $k+1\leq e^+ +e^-$.
\end{enumerate}
\end{enumerate}
\end{theorem}

 Figure
\ref{figure3} shows two examples for $n=6$.

\begin{figure}[htp]
\begin{center}
  \includegraphics[totalheight=5.3cm]{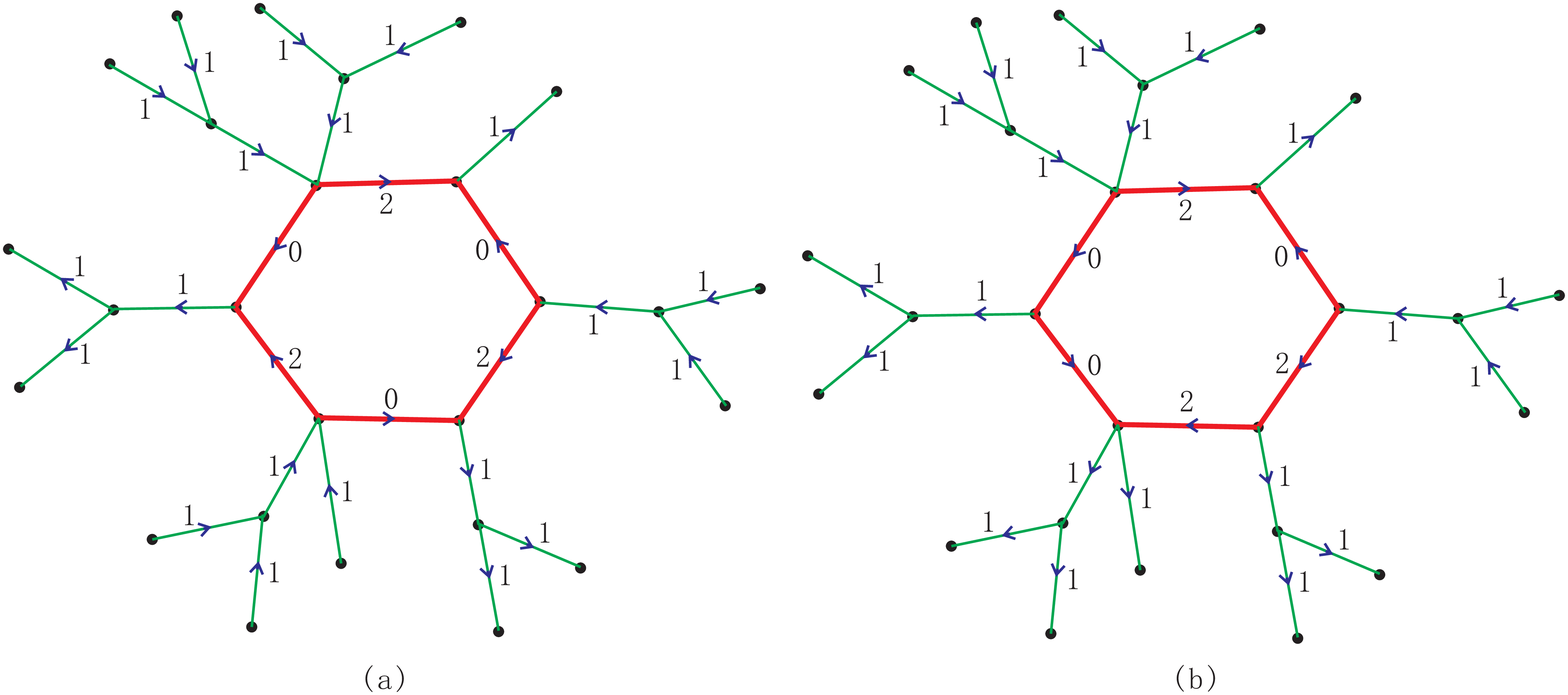}\\
  \caption{}\label{figure3}
  \end{center}
\end{figure}

\begin{proof}
\emph{The necessity}. Let $\phi_t$ be an NS flow on $S^1 \times S^2$
which satisfies that there exists an inseparable regular level set
which is homeomorphic to $S^{2}$. Suppose $L$ is a Lyapunov graph of
$\phi_t$. By (2) in Proposition \ref{proposition4.3},
$\beta_{1}(L)=1$. Therefore, there exists a unique cycle $C\subset
L$. Let $b$ be an edge in $\overline{L-C}$. By Lemma \ref{lemma4.3},
it is easy to show that the weight of $b$ is 1. Since there exists a
inseparable regular level set which is homeomorphic to $S^{2}$, by
(b) in Proposition \ref{proposition4.3}, the weight of an edge of
$C$ is 0 or 2 and there exist at least
 one  weight 0 edge and one weight 2 edge in $C$. Let  $a$ and $b$ be two
 adjacent edges in $C$. By Lemma \ref{lemma4.3} and Proposition \ref{proposition4.3}, there are three
 cases:
 \begin{enumerate}
 \item $a$ and $b$ are both weight 0 edges. Moreover, the orientations of $a$ and
 $b$ in $C$ are the same.
 \item $a$ and $b$ are both weight 2 edges. Moreover, the orientations of $a$ and
 $b$ in $C$ are the same.
 \item The weight of $a$ is 0 (2, resp.) and the weight of $b$ is 2 (0, resp.).  Moreover, the orientation of $a$
 in $C$ is opposite to the orientation of $b$ in $C$.
 \end{enumerate}
 It follows easily from the above observations that the orientation of an weight 0 edge
 in $C$ is opposite to the orientation of an weight 2 edge in $C$.

Let $V$ be a vertex in $L$ labeled with an SSFT with matrix
$A_{m\times m}$. Suppose $a$ is a weight 0 edge in $L$. We can
obtain a new Lyapunov graph $L'$ by cutting $L$ along $a$, then
pasting a singularity sink vertex $V_1$ and a singularity source
vertex $V_2$, as shown in Figure \ref{figure2}. Then $L'$ is a
Lyapunov graph of a Smale flow on $S^3$. By Theorem
\ref{theorem2.41},
  $0< \max \{e^+, e^-\} \leq k+1$, $k+1\leq \sum_{i=1}^{e^{-}}g_{i}^{-} + e^{+}$ and
$k+1\leq\sum_{j=1}^{e^{+}}g_{j}^{+} + e^{-}$. Therefore, by some
easy computations, $V$ satisfies (2) in Theorem \ref{theorem5.2}.

 \emph{The sufficiency}.
  Let $L$ be a Lyapunov graph which satisfies
(1) and (2) of Theorem \ref{theorem5.2}. Suppose $a$ is a weight 0
edge in $L$. We can obtain a new Lyapunov graph $L'$ by cutting $L$
along $a$, then pasting a singularity sink vertex $V_1$ and a
singularity source vertex $V_2$, as shown in Figure \ref{figure2}.
It is easy to show that $L'$ satisfies the sufficient condition of
Theorem \ref{theorem2.41}. Therefore, there exists a Smale flow
$\psi_t$ on $S^3$ with Lyapunov graph $L'$. Suppose the singularity
sink labeled with $V_1$ and the singularity source labeled with
$V_2$  are $s_1$ and $s_2$ respectively. Cutting two standard 3-ball
neighborhoods $N(s_1)$ and $N(s_2)$ of $s_1$ and $s_2$ respectively,
then pasting $S^3-(N(s_1)\sqcup N(s_2))$ along $\partial N(s_1)$ and
$\partial N(s_2)$ suitably, we obtain an NS flow $\phi_t$ on $S^1
\times S^2$ with Lyapunov graph $L$.
\end{proof}

The following lemma is due to J. Franks \cite{F5}:

\begin{lemma}\label{lemma5.3}
Let $A_{n\times n}$ be a nonnegative irreducible integer matrix,
then for any $N> 0$, there exists a nonnegative irreducible integer
matrix $A'=(a'_{ij})_{m\times m}$ such that:
\begin{enumerate}
\item the SSFT with $A$ is conjugate to the SSFT
with  $A'$;
\item $a'_{ij}> N$ for any $i,j \in \{1,...,m\}$ and $a'_{ij}$ is
an even number if $i\neq j$.
\end{enumerate}
\end{lemma}

\begin{lemma}\label{lemma5.4}
Let $A_{n\times n}$ be a nonnegative irreducible integer matrix. For
any $e^+, e^- \in \mathbb{N}$ with $e^+ +e^- =k$, there exists an NS
flow $\psi_t$ on a compact 3-manifold $N$ such that:
\begin{enumerate}
\item the chain recurrent set of $\psi_t$ is conjugate to the SSFT
with $A$;
\item $\psi_t$ is transverse to $\partial N$ and moreover, the entrance set and the exit set  of
$\psi_t$ on $\partial N$ are composed of $e^+$ tori and  $e^-$ tori
respectively;
\item $N$ can be embedded into $S^1 \times S^2$ by an embedding map $i:N\rightarrow S^1 \times
S^2$ such that $\overline{S^1 \times S^2 -i(N)}\cong (e^+ +e^-) (S^1
\times D^2)$.
\end{enumerate}
\end{lemma}

\begin{proof}
\emph{Step 1}. Construct a gradient-like diffeomorphism $f$ on
$S^2$.

We consider a 2-sphere $S^2$ with the spherical geometry. As Figure
\ref{figure4} shows, there exists a handle decomposition with two
0-handles, two 2-handles and two 1-handles. The handle decomposition
corresponds to a gradient-like flow $\varphi_t$ on $S^2$ with two
index 0 singularities  $a_{1}$ and  $a_{2}$, two index 2
singularities $r_{1}$ and $r_{2}$ and two index 1 singularities
$s_{1}$ and $s_{2}$. $C_1$ and $C_2$ are two orthogonal circles in
$S^2$. Figure \ref{figure4} (a) and Figure \ref{figure4} (b) both
show the handle decomposition. In Figure \ref{figure4} (b), $S^2$ is
identified with $\mathbb{C} \cup \{\infty\}$ and $r_2 =\infty$.
There exist two standard reflections $T_1$ and $T_2$ on $S^2$
leaving $C_1$ and $C_2$ fixed respectively. Set $f=T_2 \circ T_1
\circ \varphi_1: S^2\rightarrow S^2$. It is easy to check that $f$
is a gradient-like diffeomorphism on $S^2$ with an index 0 orbit
$\{a_1, a_2\}$ with periodic 2, an index 2 orbit $\{r_1, r_2\}$ with
periodic 2 and two index 1 singularities $\{ s_{1} \} \sqcup \{
s_{2} \}$.

\begin{figure}[htp]
\begin{center}
  \includegraphics[totalheight=5cm]{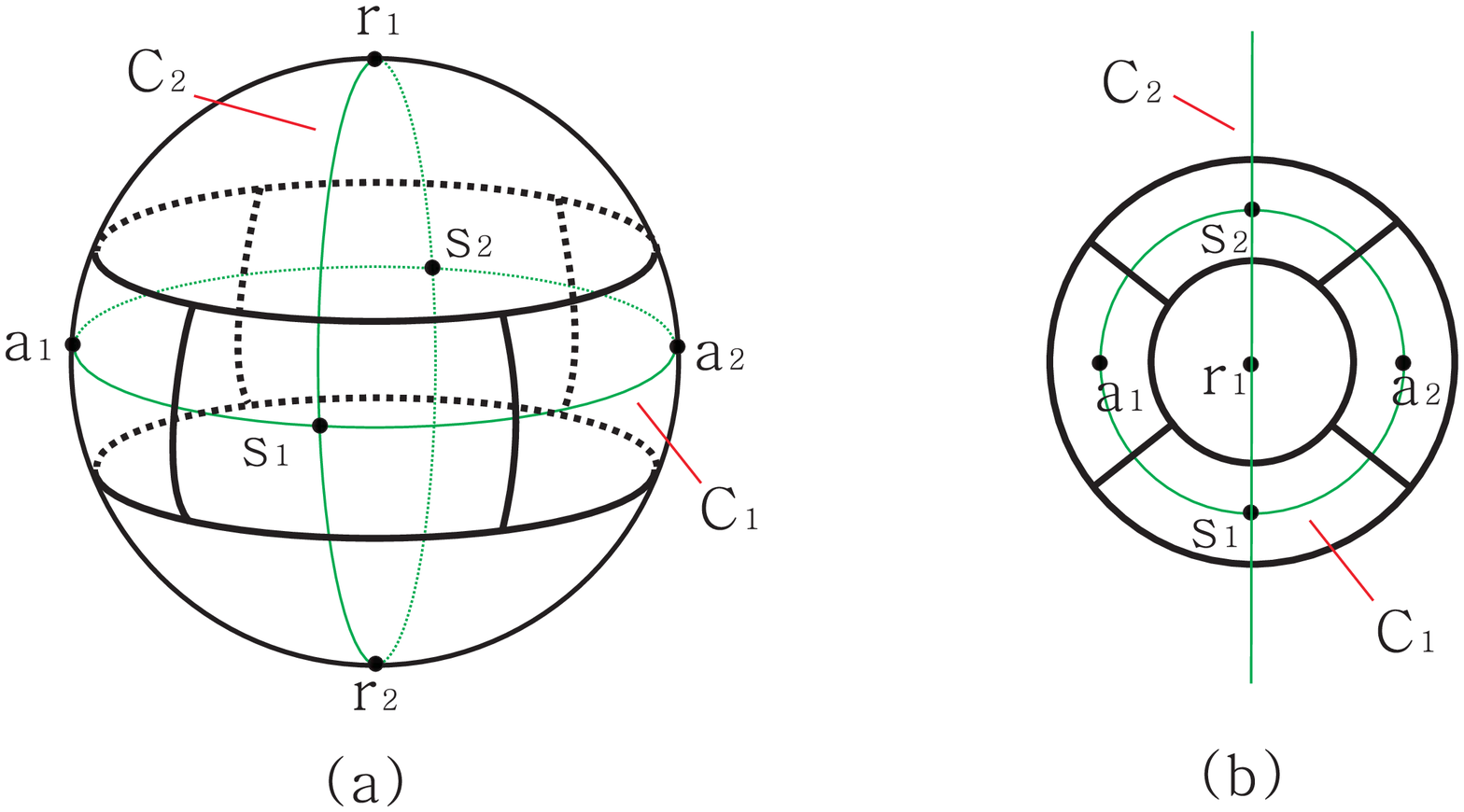}\\
  \caption{}\label{figure4}
  \end{center}
\end{figure}


\emph{Step 2}. Add handles and construct a Morse-Smale
diffeomorphism $\phi_t$ on $S^2$.

Following J. Franks, we first introduce three types of handles, as
shown in Figure \ref{figure5}. The first one is composed of a
0-handle and a 1-handle (a neighborhood of a sink and that of a
saddle) and it is called an \emph{SI handle}; the second one is
composed of a 2-handle and a 1-handle and it is called an \emph{SO
handle} (a neighborhood of a source and a saddle); the last one
doesn't admit any singularity and it is called an \emph{NI handle}
(it is called a nilpotent handle in \cite{F5}).
\begin{figure}[htp]
\begin{center}
  \includegraphics[totalheight=3.6cm]{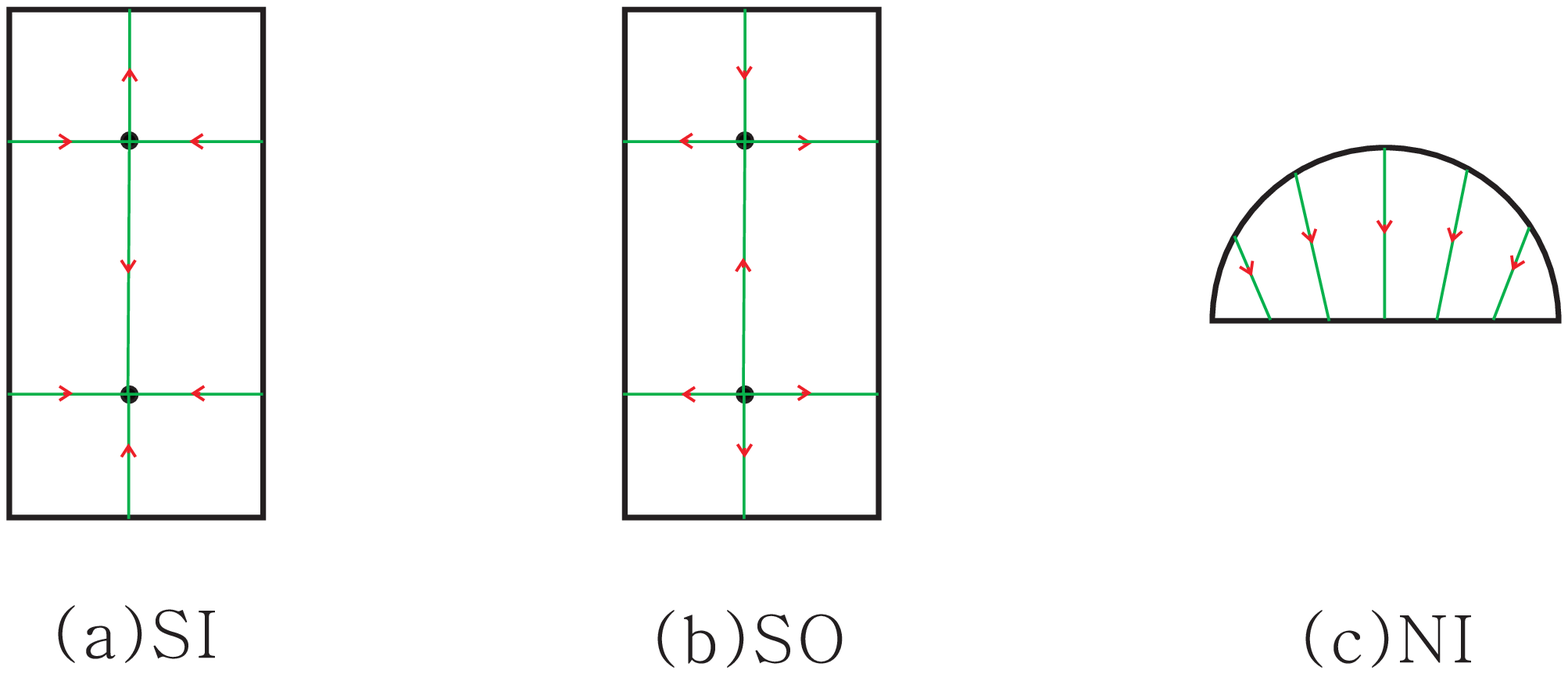}\\
  \caption{}\label{figure5}
  \end{center}
\end{figure}

Now we add handles to the handle decomposition in Step 1. In Figure
\ref{figure6}, the disk with boundary $L_3$ is denoted by $Y$. We
add $(e^+ -1)$ SO handles, $(e^- -1)$ SI handles and $(n-k)$ NI
handles to the handle decomposition in Step 1 in the interior of
$Y$. Suppose $\phi_t$ is a Morse-Smale diffeomorphism induced by the
handle decomposition after adding handles. Figure \ref{figure6}
shows an example for $e^+ =e^- =2$ and $n-k=1$.

\begin{figure}[htp]
\begin{center}
  \includegraphics[totalheight=7cm]{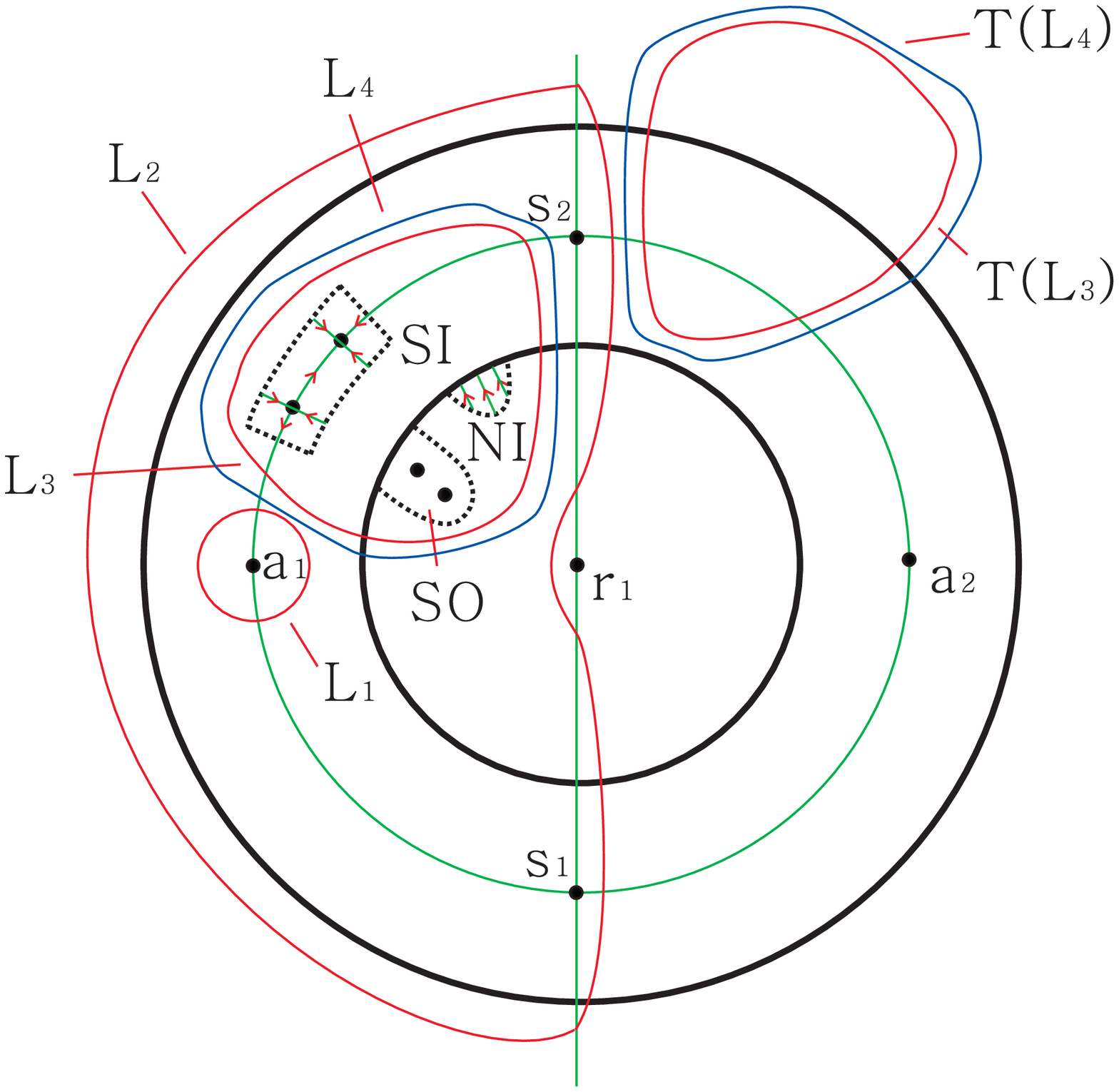}\\
  \caption{}\label{figure6}
  \end{center}
\end{figure}

\emph{Step 3}. Construct a Smale diffeomorphism $g$ on $S^2$.

Due to Lemma \ref{lemma5.3}, in order to prove Lemma \ref{lemma5.4},
we may assume that $a_{ij}$ is an even number if $i\neq j$ for any
$i,j \in \{1,...,n\}$. In Figure \ref{figure6}, we define
$X=\overline{D_2 -D_1}$ where $D_1$ ($D_2$) is the disk with
boundary $L_1$ ($L_2$). Moreover, set $X_0 = \overline{X-D}$ where
$D$ is a small standard neighborhood of the sinks and sources of the
added handles. The number of the saddle singularities in $X$ is
$(e^+ -1)+(e^- -1)+1+1=e^+ + e^- =k$ and the number of NI handles is
$n-k$. By the trick in the proof of Lemma 4.2 in \cite{F5}, we
obtain a Smale diffeomorphism $g$ on $S^2$ such that:
\begin{enumerate}
\item if $x$ is in $S^2 -X_0$ or in the right side of $C_2$ (see Figure \ref{figure4}), then $g(x)=\phi_1 (x)$;
\item there exists only one basic set $\Lambda$ in $X_0$ and $\Lambda$
is a dimension 1 saddle basic set with matrix $A$;
\item there exists a Markov partition of $\Lambda$ with $n$ rectangles $R_1 \cup ... \cup R_n$
such that:
\begin{enumerate}
\item each rectangle corresponds to a saddle singularity or an
NI handle;
\item the rectangles corresponding the saddle
singularities in $Y$ and the NI handles are in the interior of $Y$;
\item the geometric intersection matrix of the rectangles is $A$.
\end{enumerate}
\end{enumerate}

\emph{Step 4}. Construct an NS flow $\psi_t$ on a compact 3-manifold
$N$.

In Figure \ref{figure6}, the disk with boundary $L_4$ is denoted by
$Y'$ satisfying $\interior Y \subset Y'$ and $T(Y') \cap X
=\emptyset$. Here $T=T_{2} \circ T_{1}$. We define a smooth function
$h: S^2 \rightarrow \mathbb{R}$ such that:

\begin{eqnarray}\label{7}
h(x)\left\{
\begin{split}
&=0, ~if~ x\in T(Y)\\
&\in [0,1], ~if~ x\in T(Y')-T(Y)\\
&=1, ~if~ x\in \overline{S^2 -T(Y')}
\end{split}
 \right.
\end{eqnarray}

We also define a map $F:S^2\rightarrow S^2$ such that:
\begin{eqnarray}\label{8}
F(x)=\left\{
\begin{split}
&g(x), ~if~ x\in X\\
&\phi_{h(x)}(x), ~if~ x\in S^2 -X
\end{split}
\right.
\end{eqnarray}

By (\ref{7}) and the properties of $g$, it is easy to show that $F$
is a diffeomorphism on $S^2$. Define a diffeomorphism $H$ on $S^2$
by $H=T\circ F$.

By identifying $S^1\times S^2$ with the quotient space $[0,1]\times
S^2/(0,x)\sim (1,H(x))$, we define the flow $\psi'_t$ by $\psi'_t
(t_0, x)=(t_0 +t, x)$.  It is easy to verify that $\psi'_t$ is an NS
flow on $S^1\times S^2$ with $e^+$ closed orbit sources
$r^1,...,r^{e^+}$, $e^-$ closed orbit sinks $a^1,..., a^{e^-}$ and a
saddle basic set $\Lambda$ with matrix $A$. Suppose $N(r^i)$
($N(a^j)$) is a small standard neighborhood of $r^i$ ($a^j$). Denote
$\overline{S^1 \times S^2 -
\cup_{i=1}^{e^+}N(r^i)-\cup_{j=1}^{e^-}N(a^j})$ by $N$ and
$\psi'_t|_N$ by $\psi_t$. By the properties of $\psi'_t$, it is easy
to show that $\psi_t$ satisfies (1), (2) and (3) of Lemma
\ref{lemma5.4}.
\end{proof}

\begin{theorem} \label{theorem5.5}
Let $L$ be an abstract Lyapunov graph, then $L$ is associated with
an NS flow $\phi_{t}$ on $S^{1}\times S^{2}$ such that each regular
level set is separable if and only if the following conditions hold:
\begin{enumerate}

\item The underlying graph $L$ is a tree with exactly one edge attached
to each closed orbit sink or closed orbit source vertex. The weight
of each edge of $L$ is 1.

\item If a vertex is labeled with an SSFT with matrix $A_{n\times
n}$, then
\begin{equation}\label{1}
 \begin{split}
    &0< e^{+}\leq k+1, 0< e^{-}\leq k+1 \text{\ and}\\
    &k\leq e^{+}+ e^{-}.
 \end{split}\nonumber
\end{equation}
\item There is at most 1 vertex satisfying that $k= e^{+}+ e^{-}$.
\end{enumerate}
\end{theorem}

\begin{proof}
\emph{The necessity}. Let $\phi_t$ be an NS flow on $S^1 \times S^2$
such that each regular level set is separable. Suppose $L$ is a
Lyapunov graph of $\phi_t$. By (1) in Proposition
\ref{proposition4.3}, the underlying graph $L$ is a tree and the
weight of each edge of $L$ is 1. (2) and (3) follow from (2) in
Proposition \ref{proposition4.5}.

 \emph{The sufficiency}.
  Let $L$ be a Lyapunov graph which satisfies
(1), (2) and (3).

If there doesn't exist a vertex in $L$ with $k= e^{+}+ e^{-}$. Then
by Theorem \ref{theorem2.4} and the fact that we can paste $S^1
\times S^2$ up as two solid tori, it is easy to show that there
exists an NS flow $\phi_t$ on $S^1 \times S^2$ with Lyapunov graph
$L$.

If there exists a vertex in $L$ with $k= e^{+}+ e^{-}$. We cut $L$
along all the entrance edges and exit edges of $V$ and paste $e^+$
closed orbit sink vertices $V_{1}^+,...,V_{e^+}^+$ and $e^-$ closed
orbit source vertices $V_{1}^-, ...,V_{e^-}^-$ to the cut components
of $L$ without $V$ suitably (similar to the surgery in Figure
\ref{figure2}). Therefore, we obtain $k$ abstract Lyapunov graphs:
$L_{1}^+,..., L_{e^+}^+$ and $L_{1}^-,...,L_{e^-}^-$. Here $L_{i}^+$
and $L_{j}^-$ contain vertices  $V_{i}^+$ and $V_{j}^-$
respectively. It is easy to check that $L_{i}^+$ and $L_{j}^-$
satisfy the sufficient condition of Theorem \ref{theorem2.4} for any
$i\in \{1,...,e^+\}$ and $j\in \{1,...,e^-\}$. Therefore, there
exist NS flows $\varphi^{i}_t$ and $\psi^{j}_t$ with Lyapunov graphs
$L_{i}^+$ and $L_{j}^-$ respectively. Let $a_i$ and $r_j$ be the
closed orbits associated with $V_{i}^+$ and $V_{j}^-$ respectively.
By the constructions in \cite{F5}, we can suppose that $a_i$ and
$r_j$ are all trivial knots.

For $e^+$, $e^-$ and $A$, we choose a compact 3-manifold $N$ and an
NS flow $\psi_t$ in Lemma \ref{lemma5.4}. Cut a standard
neighborhood $N(a_i)$ of $a_i$ and a standard neighborhood $N(r_j)$
of $r_j$ and paste $S^3 - N(a_i)$ and $S^3 -N(r_j)$ to an entrance
set of $\partial N$ and an exit set of $\partial N$ respectively.
Doing the surgeries above for any $i\in \{1,...,e^+\}$ and $j\in
\{1,...,e^-\}$, we obtain a closed 3-manifold $M$.  $a_i$ and $r_j$
are all trivial knots. Moreover,
 $N$ can be embedded into $S^1 \times S^2$ and there exists an embedding map $i:N\rightarrow S^1 \times
S^2$ such that $\overline{S^1 \times S^2 -i(N)}\cong (e^+ +e^-) (S^1
\times D^2)$. Therefore, if the surgeries are suitable, then:
\begin{enumerate}
\item $M\cong S^1 \times S^2$;
\item $\varphi^{i}_t$ on $S^3 - N(a_i)$, $\psi^{j}_t$ on $S^3
-N(r_j)$ and $\psi_t$ on $N$ form an NS flow $\phi_t$ on $S^1 \times
S^2$ with Lyapunov graph $L$
\end{enumerate}
\end{proof}

\section{Singular Vertex}\label{section6}
Basic definitions and facts about 3-manifolds can be found in
\cite{He}.

A \emph{filtrating neighborhood} \cite{BB} of a dimension 1 basic
set $K$ of an NS flow on a 3-manifold is a neighborhood $U$ such
that:
\begin{enumerate}

\item $K$ is the maximal invariant set (for $\phi_{t}$) in $U$.
 \item The intersection of any orbit of $\phi_{t}$ with $U$ is
 connected (in other terms, any orbit getting out of $U$ never comes
 back).
\end{enumerate}

In \cite{BB}, F. B\'eguin and C. Bonatti proved that for a given
dimension 1 basic set $K$, the filtrating neighborhood of $K$ is
unique up to topological equivalence. For an NS flow $\phi_t$ on a
3-manifold $M$, let $g:M\rightarrow R$ be a Lyapunov function
associated to $\phi_t$ and $L$ be a Lyapunov graph associated with
$g$. Suppose $K$ is a basic set of $\phi_t$, $g(K)=y_0$ and $K$
corresponds to the vertex $v\in L$. Then it is easy to show that the
connected component  $U\subset g^{-1}(y_0 - \epsilon, y_0 +
\epsilon)$ which contains $K$ is the filtrating neighborhood of $K$.
This observation  indicates that the topological structures of all
Lyapunov graphs of an NS flow on a 3-manifold are the same. We also
call the filtrating neighborhood of $K$ the filtrating neighborhood
associated with $v$.

\begin{definition}
Let $L$ be a Lyapunov graph associated with an NS flow on a closed
orientable 3-manifold $M$. If a vertex $v$ of $L$ satisfies that the
filtrating neighborhood associated with $v$ is not homeomorphic to a
link complement, then we call the vertex $v$ a \emph{singular
vertex}. Otherwise, we call the vertex $v$ a \emph{nonsingular
vertex}.
\end{definition}

\begin{remark}
Oka \cite{Ok} also introduced the concept singular vertex. But our
definition  differs from his.
\end{remark}

By a similar proof as in Proposition \ref{proposition4.5} (1), it is
easy to show the following proposition.

\begin{proposition}\label{proposition6.3}
Let $L$ be a Lyapunov graph associated with an NS flow on a closed
orientable 3-manifold $M$. If $V$ is a nonsingular vertex, then
\begin{equation}
 \begin{split}
    &0< e^{+}\leq k+1, 0< e^{-}\leq k+1 \text{\ and}\\
    &k+1\leq e^{+}+ e^{-}.
 \end{split}\nonumber
\end{equation}
\end{proposition}

Proposition \ref{proposition6.3} tells us that if a vertex is a
nonsingular vertex, the dynamics of the filtrating neighborhood
associated with the vertex is similar to the dynamics of a
filtrating neighborhood of an NS flow on $S^3$. Therefore, we should
study the filtrating neighborhood associated with a singular vertex.

\begin{example}
\begin{enumerate}
\item Obviously, any vertex of a Lyapunov graph of an NS flow on
$S^3$ is a nonsingular vertex.
\item The vertex $v$ corresponds to  the filtrating neighborhood $N$
and the flow $\psi_t$ in Lemma \ref{lemma5.4} is a singular vertex.
In fact, for such a vertex $v$, by Lemma \ref{lemma5.4}, $e^+ + e^-
=k$. Therefore, $v$ doesn't satisfy the necessary condition of a
nonsingular vertex in Proposition \ref{proposition6.3}.
\item For NS flows on $S^1
\times S^2$, by Theorem \ref{theorem5.1}, Theorem \ref{theorem5.2}
and Theorem \ref{theorem5.5}, we have:
\begin{enumerate}
  \item  Let $L$ be a Lyapunov graph associated with an NS flow on $S^1
\times S^2$. If the weight of any edge of $L$ is 1, then there
exists at most 1 singular vertex in $L$;
  \item For any $n\in \mathbb{N}$, there exists an NS
flow $\phi_t$ such that there exist $n$ singular vertices in $L$.
Here $L$ is a Lyapunov graph associated with $\phi_t$.
\end{enumerate}
\end{enumerate}
\end{example}

Now we consider the number of the singular vertices in a Lyapunov
graph associated with an NS flow on an irreducible, closed
orientable 3-manifold.

By Dehn's lemma and some combinatorial surgeries, it is easy to show
the following lemma.

\begin{lemma} \label{lemma6.4}
Let $T$ be a compressible torus in an irreducible, closed orientable
3-manifold $M$. Then $T$ bounds a knot complement in $M$.
\end{lemma}

\begin{proposition}
Let $L$ be a Lyapunov graph associated with an NS flow $\phi_t$ on
an irreducible, closed orientable 3-manifold $M$. Then there exist
at most $H(M)+1$ singular vertices
 in $L$ where $H(M)$ is the Haken number of $M$.
\end{proposition}

\begin{proof}
 For each edge $e$ of $L$, we choose a regular level set $\Sigma_e$ associated with $e$.
 Set $\sum=\{\Sigma_e | e\in L\}$.
 By Theorem \ref{theorem2.6}, the weight of each edge of $L$ is 1.
 Therefore, each $\Sigma_e$ is homeomorphic to $T^2$.

 Let $\mathscr{T}=\{\mathcal {T}|\mathcal {T}=
 \{T_1, T_2, ... ,T_m\}\subset\sum\}$ where each $T_i$ is incompressible and $T_i$ is not parallel to $T_j$
 for any $i,j\in \{1,...,m\}$ and $i\neq j$. Obviously, $m\leq
 H(M)$. For any $\mathcal {T}=
 \{T_1, T_2, ... ,T_m\} \in \mathscr{T}$, let $d(\mathcal {T})$ be the number of the connected
 components of $M- (T_1\sqcup T_2\sqcup ...  \sqcup T_m)$. Obviously, $d(\mathcal {T})\leq m+1 \leq
 H(M)+1$. Therefore, there exists $\mathcal{T}_0 =\{T_1, T_2, ... ,T_{m_0}\}\in
 \mathscr{T}$ satisfying $d(\mathcal{T}_0)=\max_{\mathcal {T} \in
 \mathscr{T}}(d(\mathcal{T}))\leq
 H(M)+1$. Let $M_1, ...,M_{d(\mathcal{T}_0)}$ be the components
 of $M- (T_1\sqcup T_2\sqcup ...  \sqcup T_{m_0})$. For any
 $\Sigma_e \in \sum-\mathcal{T}_0$, there are exactly the following two cases.
 \begin{enumerate}
 \item If $\Sigma_e$ is incompressible in $M$, then $\Sigma_e$ is parallel to $T_i$
 for some $i\in \{1,...,m\}$. Therefore, $\Sigma_e$ and $T_i$ bound a
 compact 3-manifold $N_i$ which is homeomorphic to $T^2 \times
 [0,1]$.
 \item If $\Sigma_e$ is compressible in $M$, by Lemma \ref{lemma6.4}, then $\Sigma_e$ bounds a knot complement $N\subset M$.
 \end{enumerate}
 Suppose $\Sigma_e \subset M_j$ for some $j \in
 \{1,...,d(\mathcal{T}_0)\}$. In either of the above two cases, $\Sigma_e$
 divides $M_j$ into two parts $M_{j}^1$ and $M_{j}^2$ such that one of $M_{j}^1$ and $M_{j}^2$ is a link component.
 Therefore, by some trivial inductive arguments, we can show that in $M_i$
 for any $i\in \{1,...,d(\mathcal{T}_0)\}$, there exists at most 1 filtrating
 neighborhood associated with a singular vertex. Noticing that $M_{d(\mathcal{T}_0)}\leq
 H(M)+1$, we have that there exist at most $H(M)+1$ singular vertices
 in $L$.
\end{proof}

\section*{Acknowledgments} The author would like to thank Hao Yin for
his many helpful suggestions and comments.

\bibliographystyle{amsplain}

\end{document}